\documentclass[oneside,english,reqno]{amsart}
\usepackage[T1]{fontenc}
\usepackage[utf8]{inputenc}
\usepackage{geometry}
\usepackage{amssymb}
\usepackage{mathtools}
\usepackage{csquotes}
\usepackage[backend=biber,sorting=none,doi=false,isbn=false,url=false]{biblatex}
\usepackage[hidelinks]{hyperref}
\hypersetup{pdftitle={Convergence of the free energy of vector spin glasses with non-convex interactions}}

\numberwithin{equation}{section}
\theoremstyle{plain}
\newtheorem{thm}{Theorem}[section]
\newtheorem{prop}[thm]{Proposition}
\newtheorem{lem}[thm]{Lemma}
\newtheorem{cor}[thm]{Corollary}
\theoremstyle{remark}
\newtheorem{rem}[thm]{Remark}

\newcommand{\E}{\mathbb{E}}
\newcommand{\R}{\mathbb{R}}
\newcommand{\sfp}{\mathsf{p}}
\newcommand{\sfq}{\mathsf{q}}
\newcommand{\overlap}{\mathsf{R}}
\newcommand{\dd}{\mathop{}\!\mathrm{d}}
\DeclareMathOperator{\supp}{supp}
\DeclareMathOperator{\tr}{tr}
\DeclareMathOperator{\Var}{Var}
\DeclareMathOperator{\Cov}{Cov}
\allowdisplaybreaks

\title[Convergence of vector-spin free energies]{Convergence of the free energy of vector spin glasses with non-convex interactions}
\author{Fu-Hsuan Ho}
\address{Department of Mathematics, Weizmann Institute of Science, 76100 Rehovot, Israel}
\email{fu-hsuan.ho@weizmann.ac.il}
\date{}

\begin{document}
\begin{abstract}
We prove that the enriched free energy of vector spin glasses converges
as the number of sites tends to infinity. The interaction is not assumed
to be convex. This is the convergence assumed by Chen and Mourrat (Probab.\ Math.\ Phys., 2025) in
their representation of the limit as a critical point representation, 
so that representation holds without further assumptions.
\end{abstract}
\maketitle

\section{Introduction}
\label{sec:intro}

We study the enriched free energy of vector spin glasses without assuming
convexity of the interaction. Our aim is to prove the convergence required
in the critical-point representation of Chen and Mourrat
\cite{ChenMourrat2025}.

Fix integers $D,N\geq1$. Let $P_1$ be a compactly supported probability
measure on $\R^D$, and set $P_N=P_1^{\otimes N}$. We assume that
$\xi:\R^{D\times D}\to\R$ satisfies $\xi(0)=0$ and admits an absolutely
convergent power-series expansion on $\R^{D\times D}$. The Hamiltonian
$(H_N(\sigma))_{\sigma\in\R^{D\times N}}$ of the vector spin glass model
is a centered Gaussian process with covariance
\begin{equation}
  \E\bigl[H_N(\sigma)H_N(\tau)\bigr]
  =N\xi\left(N^{-1}\sigma\tau^\intercal\right).
  \label{eq:covariance}
\end{equation}

Denote by $S^D$ and $S^D_+$ the spaces of symmetric and positive
semidefinite $D\times D$ matrices. On $\R^{D\times D}$, and more
generally on $\R^{D\times N}$, we use the Frobenius inner product
$a\cdot b=\tr(ab^\intercal)$ and norm $|a|=(a\cdot a)^{1/2}$. We write
$\|a\|_{\mathrm{op}}$ for the operator norm, and $a\preceq b$ if
$b-a$ is positive semidefinite.

Let $\mathcal Q$ be the set of right-continuous paths
$\sfq:[0,1)\to S^D_+$ that are increasing for $\preceq$, and set
$\mathcal Q_r=\mathcal Q\cap L^r([0,1];S^D)$ for $1\leq r\leq\infty$.
The norm $|\sfq|_{L^r}$ is taken with respect to $|\cdot|$. For
$\sfq\in\mathcal Q_\infty$, write $\sfq(1)=\lim_{u\uparrow1}\sfq(u)$.

For $t\geq0$ and $\sfq\in\mathcal Q_\infty$, define the enriched
Hamiltonian and the enriched free energy by
\begin{align}
  H_N^{t,\sfq}(\sigma,\alpha)
  &=\sqrt{2t}\,H_N(\sigma)
    -tN\xi\left(N^{-1}\sigma\sigma^\intercal\right)
    +\sqrt2\,W_N^{\sfq}(\alpha)\cdot\sigma
    -\sfq(1)\cdot\sigma\sigma^\intercal,
    \label{eq:corrected-hamiltonian}\\
  \bar F_N(t,\sfq)
  &=-\frac1N\E\log
      \iint_{\R^{D\times N}\times\mathfrak U}
    \exp\bigl(H_N^{t,\sfq}(\sigma,\alpha)\bigr)
    \dd P_N(\sigma)\dd\mathfrak R(\alpha).
    \label{eq:free-energy}
\end{align}
Here $\mathfrak R$ is the Ruelle probability cascade whose overlap
distribution is uniform on $[0,1]$, realized as a random probability
measure on the unit sphere of a separable Hilbert space $\mathfrak H$.
We write $\mathfrak U=\supp\mathfrak R$ and
$\alpha\wedge\alpha'=\langle\alpha,\alpha'\rangle_{\mathfrak H}$.
Conditionally on $\mathfrak R$, the columns of $W_N^{\sfq}$ are
independent centered Gaussian processes indexed by $\mathfrak U$ with
covariance $\sfq(\alpha\wedge\alpha')$, and they are independent of
$H_N$. By \eqref{eq:path-continuity} below, $\bar F_N(t,\cdot)$ is
Lipschitz for the $L^1$ norm on $\mathcal Q_\infty$, and it extends by
continuity to $\mathcal Q_1$.

We now state the main result of the present paper.

\begin{thm}
\label{thm:convergence}
There exists a function $f:\R_+\times\mathcal Q_1\to\R$
such that for every $t\geq 0$ and $\sfq\in\mathcal Q_1$,
\begin{equation}
  \lim_{N\to\infty}\bar F_N(t,\sfq)=f(t,\sfq).
  \label{eq:convergence}
\end{equation}
\end{thm}

Theorem~\ref{thm:convergence} is the convergence assumption preceding
Theorem~1.2 in \cite{ChenMourrat2025}. To state its consequence, we
recall some notation. Taking $N=1$ and $t=0$ in \eqref{eq:free-energy},
set
\[
  \psi(\sfq)=\bar F_1(0,\sfq),
  \qquad
  \mathcal J_{t,\sfq}(\sfq',\sfp)
  =\psi(\sfq')+\langle\sfp,\sfq-\sfq'\rangle_{L^2}
    +t\int_0^1\xi(\sfp(u))\dd u.
\]

\begin{cor}
\label{cor:critical-point}
For every $t\geq0$ and $\sfq\in\mathcal Q_2$, there exist
$\sfq'\in\mathcal Q_2$ and $\sfp\in\mathcal Q_\infty$ such that
\begin{equation}
  \sfp=\partial_{\sfq}\psi(\sfq'),\qquad
  \sfq'=\sfq+t\nabla\xi(\sfp),\qquad
  f(t,\sfq)=\mathcal J_{t,\sfq}(\sfq',\sfp).
  \label{eq:critical-point}
\end{equation}
The path $\sfq'-\sfq$ is bounded. If $\sfq$ is bounded, so is $\sfq'$.
\end{cor}

\begin{proof}
By Theorem~\ref{thm:convergence}, the hypothesis of
\cite[Theorem~1.2]{ChenMourrat2025} holds. Proposition~7.3 of
\cite{ChenMourrat2025}, together with the proof of Theorem~1.2 there,
gives a bounded path $\sfp\in\mathcal Q_\infty$ such that, with
$\sfq'=\sfq+t\nabla\xi(\sfp)$, the first and last identities in
\eqref{eq:critical-point} hold. Since $\nabla\xi$ is increasing on
$S^D_+$ for every $\xi$ as in \eqref{eq:covariance}
\cite{ChenMourrat2025}, the path $\sfq'$ is increasing. Since $\sfp$ is
bounded and $\nabla\xi$ is continuous, the path $\sfq'-\sfq$ is
bounded. Hence $\sfq'\in\mathcal Q_2$, and $\sfq'$ is bounded whenever
$\sfq$ is.
\end{proof}

\subsection{Related work}
\label{sec:related-work}

When $\xi$ is not convex, Chen and Mourrat showed that the limit free
energy, if it exists, is a critical value of the functional
$\mathcal J_{t,\sfq}$ \cite{ChenMourrat2025}; Theorem~\ref{thm:convergence}
removes this assumption for vector spins with a compactly supported
reference measure. Using the Hamilton--Jacobi approach, the limit free
energy of non-convex models has recently been identified for multi-species
models with centered $\pm1$ spins \cite{ChenIssaMourrat2026} and for
multi-species spherical models \cite{ChenMourrat2026spherical}.

Our argument does not identify the limit beyond
Corollary~\ref{cor:critical-point}, but it applies to general vector
reference measures with compact support and gives a rate when the
support is finite.

\subsection*{Organization}

In Section~\ref{sec:reduction}, we show by rescaling and discretization
that it suffices to consider reference measures with finite support in
the unit ball and bounded paths. For such models, we prove that the
free energies of sizes $n$ and $kn$ are close, with an explicit rate.
In Section~\ref{sec:overlap-reduction}, we introduce the interpolation
between the model of size $kn$ and $k$ blocks with independent Hamiltonians of size $n$,
restricted to balanced configurations. We show that the restriction
costs little at both ends of the interpolation. We then recall the
Parisi equation and its diffusion, and use them to reduce the
comparison to a bound on the difference between the overlaps of two
blocks. In Section~\ref{sec:overlap-proof}, we prove this bound. The
fluctuations of the block overlaps are controlled by the Hessian of the
Parisi solution. Their conditional means are controlled by the
Brascamp--Lieb argument of Aronow and Lopatto \cite{AronowLopatto2026}, combined with the
exchange of two blocks of a balanced configuration.

\subsection*{Acknowledgments}

A large language model was used to understand the literature and to
develop the proof strategy. In particular, it was used to understand
Aronow and Lopatto's recent work
\cite{AronowLopatto2026}. 

FHH acknowledges support from the ISF grant (No.~2055/21) and the ERC
grant (No.~101165541, Horizon Europe). FHH also thanks Michael
Hofstetter and Justin Ko for stimulating discussions.

\section{Reduction to a Cauchy sequence}
\label{sec:reduction}

We prove the convergence by showing that, for each $(t,\sfq)$, the
sequence $(\bar F_N(t,\sfq))_{N\geq1}$ is Cauchy. We first consider a
reference measure with finite support in the unit ball and a bounded path. The main
estimate is the following proposition.

\begin{prop}
\label{prop:rate}
Suppose that $P_1$ has finite support in the unit ball. For every
$T,Q<\infty$, there exists a constant $C=C_{P_1,D,\xi,T,Q}>0$ such that
for all $n\geq1$ and $k\geq2$,
\begin{equation}
  \sup_{\substack{0\leq t\leq T,\ \sfq\in\mathcal Q_\infty\\
                   |\sfq|_{L^\infty}\leq Q}}
  \bigl|\bar F_{kn}(t,\sfq)-\bar F_n(t,\sfq)\bigr|
  \leq Cn^{-1/4}.
  \label{eq:rate}
\end{equation}
\end{prop}

\begin{rem}
\label{rem:rate}
Letting $k\to\infty$ in \eqref{eq:rate} and using
Theorem~\ref{thm:convergence}, we obtain
$|\bar F_n(t,\sfq)-f(t,\sfq)|\leq Cn^{-1/4}$ in the same range of
$(t,\sfq)$. By the rescaling in Section~\ref{sec:proof-main}, the same
bound holds for every reference measure with finite support, with a constant
depending also on the radius of the support.
\end{rem}

We deduce Theorem~\ref{thm:convergence} from Proposition~\ref{prop:rate}
by rescaling and approximation.

\subsection{Proof of Theorem~\ref{thm:convergence}
assuming Proposition~\ref{prop:rate}}
\label{sec:proof-main}

Let $r>0$ be such that $\supp P_1$ is contained in the closed
ball $B(0,r)$ of $\R^D$. Fix $\rho>0$ and a finite $\rho$-net
$\{\tau_1,\ldots,\tau_J\}\subseteq\supp P_1$. Define
$T_\rho(x)=\tau_{j(x)}$ for $x\in\supp P_1$, where $j(x)$ is
the smallest index minimizing $|x-\tau_j|$. The map $T_\rho$ is
measurable and $|T_\rho(x)-x|\leq\rho$. For
$\sigma\in(\supp P_1)^N$, we write $T_\rho\sigma$ for the
configuration obtained by applying $T_\rho$ to each column of
$\sigma$, and we set
\[
 F^\rho_N(t,\sfq)
 =-\frac1N\E\log\iint
 \exp\bigl(H_N^{t,\sfq}(T_\rho\sigma,\alpha)\bigr)
 \dd P_N(\sigma)\dd\mathfrak R(\alpha).
\]
The law of $r^{-1}T_\rho(x)$ under $P_1(\dd x)$ is
\[
 \widetilde P_1
 =\sum_{j=1}^{J}
 P_1\bigl(T_\rho^{-1}(\{\tau_j\})\bigr)
 \delta_{r^{-1}\tau_j},
\]
which has finite support in the unit ball. Since
$W_N^{r^2\sfq}$ has the same law as $r W_N^{\sfq}$, the
quantity $F^\rho_N(t,\sfq)$ is the enriched free energy with reference measure
$\widetilde P_1$, interaction $a\mapsto\xi(r^2a)$, and path
$r^2\sfq$. This interaction also has an absolutely convergent
power series.

We compare $\bar F_N(t,\sfq)$ with $F^\rho_N(t,\sfq)$. Both are
integrals against $P_N\otimes\mathfrak R$; only the configuration in
the Hamiltonian changes. For $\sigma,\tau\in(\supp P_1)^N$, all columns
of $\sigma,\tau,T_\rho\sigma,T_\rho\tau$ have norm at
most $r$. Consequently,
\[
 \left|
 \frac{\sigma\tau^\intercal}{N}
 -\frac{(T_\rho\sigma)(T_\rho\tau)^\intercal}{N}
 \right|
 \leq 2r\rho,
\]
and all these overlap matrices lie in the ball of radius $r^2$
in $\R^{D\times D}$. Conditionally on the cascade, the covariance
kernels of the centered parts of the two Hamiltonians therefore differ
by at most
\[
 4Nr\rho\Bigl(
 t\sup_{|a|\leq r^2}|\nabla\xi(a)|
 +|\sfq|_{L^\infty}
 \Bigr).
\]
In both Hamiltonians, the non-random terms are minus one half of the
variance of the centered part. Along the Gaussian interpolation between
the two Hamiltonians, these terms cancel the diagonal terms, and the
derivative of the free energy is $(2N)^{-1}$ times an average of the
covariance difference at two replicas. Hence
\begin{equation}
 \sup_{N\geq1}
 \bigl|\bar F_N(t,\sfq)-F^\rho_N(t,\sfq)\bigr|
 \leq
 2r\rho
 \Bigl(
 t\sup_{|a|\leq r^2}|\nabla\xi(a)|
 +|\sfq|_{L^\infty}
 \Bigr).
 \label{eq:prior-approximation}
\end{equation}

Fix $\rho>0$. Proposition~\ref{prop:rate}, applied to
$\widetilde P_1$, the interaction $a\mapsto\xi(r^2a)$, and the
path $r^2\sfq$, gives a constant $C_\rho$ such that for
$n,m\geq2$,
\[
 |F^\rho_n(t,\sfq)-F^\rho_m(t,\sfq)|
 \leq|F^\rho_n(t,\sfq)-F^\rho_{nm}(t,\sfq)|
 +|F^\rho_{nm}(t,\sfq)-F^\rho_m(t,\sfq)|
 \leq C_\rho\bigl(n^{-1/4}+m^{-1/4}\bigr).
\]
Thus $(F^\rho_N(t,\sfq))_{N\geq1}$ is Cauchy. The triangle
inequality and \eqref{eq:prior-approximation} give
\[
 \limsup_{n,m\to\infty}
 \bigl|\bar F_n(t,\sfq)-\bar F_m(t,\sfq)\bigr|
 \leq
 4r\rho\Bigl(
 t\sup_{|a|\leq r^2}|\nabla\xi(a)|
 +|\sfq|_{L^\infty}
 \Bigr).
\]
Letting $\rho\downarrow0$ proves the Cauchy property for every
$t\geq0$ and $\sfq\in\mathcal Q_\infty$. The dependence of $C_\rho$
on $\rho$ is harmless, since $\rho$ is fixed before $n,m$ tend to
infinity.

Finally, by \cite[Proposition~5.1]{ChenMourrat2025}, uniformly in $N$,
\begin{equation}
  \bigl|\bar F_N(t,\sfq)-\bar F_N(t,\sfq')\bigr|
  \leq
  r^2|\sfq-\sfq'|_{L^1},
  \qquad \sfq,\sfq'\in\mathcal Q_1.
  \label{eq:path-continuity}
\end{equation}
For $\sfq\in\mathcal Q_1$ and $S \in[0,1)$, set $\sfq^S(s)=\sfq(s\wedge S)$.
Then $\sfq^S$ is bounded by $|\sfq(S)|$, and
$|\sfq-\sfq^S|_{L^1}\to0$ as $S\uparrow1$, since
$(1-S)|\sfq(S)|\leq\int_S^1|\sfq(s)|\dd s$. Convergence for bounded
paths and \eqref{eq:path-continuity} give
\[
\limsup_{n,m\to\infty}\bigl|\bar F_n(t,\sfq)-\bar F_m(t,\sfq)\bigr|
\leq
2r^2|\sfq-\sfq^S|_{L^1}.
\]
Letting $S\uparrow1$ proves the Cauchy property for every $t\geq0$ and
$\sfq\in\mathcal Q_1$.

\section{Reduction to control of an interpolated model}
\label{sec:overlap-reduction}

Throughout this section, $P_1$ has finite support in the unit ball.
Write $\sigma^{(\ell)}$ for the $\ell$-th block of $n$ columns of
$\sigma\in\R^{D\times kn}$. Let $H_n^{(1)},\ldots,H_n^{(k)}$ be
independent copies of $H_n$, independent of $H_{kn}$, and set
\begin{equation}
 H_{\theta,k,n}(\sigma)=\sqrt\theta\,H_{kn}(\sigma)
 +\sqrt{1-\theta}\sum_{\ell=1}^kH_n^{(\ell)}(\sigma^{(\ell)}),
 \qquad 0\leq\theta\leq1.
 \label{eq:interpolation}
\end{equation}
For every configuration, define
\begin{equation}
 H_{\theta,k,n}^{t,\sfq}(\sigma,\alpha)
 =\sqrt{2t}\,H_{\theta,k,n}(\sigma)
 -t\E[H_{\theta,k,n}(\sigma)^2]
 +\sqrt2\,W_{kn}^{\sfq}(\alpha)\cdot\sigma
 -\sfq(1)\cdot\sigma\sigma^\intercal.
\label{eq:restricted-hamiltonian}
\end{equation}
The variance is taken with $\sigma$ fixed:
\[
 \E[H_{\theta,k,n}(\sigma)^2]
 =\theta kn\,\xi\left(\frac{\sigma\sigma^\intercal}{kn}\right)
 +(1-\theta)n\sum_{\ell=1}^k
 \xi\left(\frac{\sigma^{(\ell)}(\sigma^{(\ell)})^\intercal}{n}\right).
\]

For a deterministic Borel set $\Sigma$ with $P_{kn}(\Sigma)>0$, define
\begin{equation}
 \bar F_{\theta,k,n}(t,\sfq;\Sigma)
 =-\frac1{kn}\E\log
 \iint_{\Sigma\times\mathfrak U}
 \exp(H_{\theta,k,n}^{t,\sfq}(\sigma,\alpha))
 \dd P_{kn}(\sigma)\dd\mathfrak R(\alpha).
 \label{eq:restricted-free-energy}
\end{equation}
Write $\bar F_N(t,\sfq;\Sigma)$ for the analogous restriction of
\eqref{eq:free-energy}. The restricted reference measure is not normalized. When
the restriction is omitted, $\Sigma$ is the whole configuration space.

We shall use the balanced set
\[
 \Sigma^{\mathrm{bal}}_{kn}
 =\{\sigma:L_1(\sigma)=\cdots=L_k(\sigma)\},
 \qquad
 L_\ell(\sigma)=\frac1n\sum_{i=(\ell-1)n+1}^{\ell n}
 \delta_{\sigma_{\cdot,i}}.
\]
For a balanced configuration and two of its blocks, matching identical
spins in the two blocks gives a permutation of the sites that exchanges
the two blocks and fixes the configuration. 
This is the symmetry used
in Section~\ref{sec:overlap-proof}. Balance also makes the block
self-overlaps equal to $\sigma\sigma^\intercal/(kn)$, so on this set
the variance correction in \eqref{eq:restricted-hamiltonian} is
$-tkn\,\xi(\sigma\sigma^\intercal/(kn))$ for every $\theta$.

Let $\mathcal L_n$ be the set of empirical distributions of $n$ spins
in $\supp P_1$, and put
\[
 \Sigma_n(L)
 =\left\{\sigma\in(\supp P_1)^n:
       \frac1n\sum_{i=1}^n\delta_{\sigma_{\cdot,i}}=L\right\},
 \qquad
 \Sigma^{\mathrm{bal}}_{kn}(L)=\Sigma_n(L)^k.
\]
On $(\supp P_1)^{kn}$, the balanced set is the disjoint union of the
sets $\Sigma^{\mathrm{bal}}_{kn}(L)$ over $L\in\mathcal L_n$.

We use the following continuity estimate.

\begin{lem}
\label{lem:restricted-continuity}
Let $\Sigma$ be a Borel set with $P_{kn}(\Sigma)>0$, and let
$\sfq,\sfq'\in\mathcal Q_\infty$. Then
\begin{equation}
 |\bar F_{\theta,k,n}(t,\sfq;\Sigma)
   -\bar F_{\theta,k,n}(t,\sfq';\Sigma)|
 \leq\int_0^1\|\sfq(s)-\sfq'(s)\|_{\mathrm{op}}\dd s.
 \label{eq:restricted-path-continuity}
\end{equation}
The same estimate holds for $\bar F_N(t,\cdot\,;\Sigma)$.
\end{lem}

\begin{proof}
We follow the proof of \cite[Proposition~5.1]{ChenMourrat2025}. For
$\lambda\in[0,1]$, set $\sfq_\lambda=\sfq'+\lambda\kappa$ with
$\kappa=\sfq-\sfq'$, so that $\sfq_\lambda\in\mathcal Q_\infty$.
Conditionally on the cascade, the covariance of
$\sqrt2\,W_{kn}^{\sfq_\lambda}(\alpha)\cdot\sigma$ and
$\sqrt2\,W_{kn}^{\sfq_\lambda}(\alpha')\cdot\sigma'$ has derivative
$2\kappa(\alpha\wedge\alpha')\cdot\sigma\sigma'^\intercal$ in
$\lambda$, and $-\sfq_\lambda(1)\cdot\sigma\sigma^\intercal$ is minus
one half of the variance. 
Let $\langle\cdot\rangle_\lambda$ be the Gibbs measure on
$\Sigma\times\mathfrak U$ with path $\sfq_\lambda$, and
$(\sigma^1,\alpha^1)$, $(\sigma^2,\alpha^2)$ are two independent
replicas. 
Then, the Gaussian integration by parts gives
\[
 \frac{\dd}{\dd\lambda}\bar F_{\theta,k,n}(t,\sfq_\lambda;\Sigma)
 =\frac1{kn}\E\bigl\langle
 \kappa(\alpha^1\wedge\alpha^2)\cdot\sigma^1(\sigma^2)^\intercal
 \bigr\rangle_\lambda.
\]
The columns of $\sigma^1,\sigma^2$ have norm at most $1$, so
the trace norm of $\sigma^1(\sigma^2)^\intercal$ is at most $kn$ and
\[
 \bigl|\kappa(\alpha^1\wedge\alpha^2)\cdot\sigma^1(\sigma^2)^\intercal\bigr|
 \leq kn\,\|\kappa(\alpha^1\wedge\alpha^2)\|_{\mathrm{op}}.
\]
Conditionally on the Hamiltonian, the logarithm of the spin integral
over $\Sigma$ is a Lipschitz function of the field
$W_{kn}^{\sfq_\lambda}(\alpha)$. The invariance of the cascade in
\cite[Proposition~4.8]{ChenMourrat2025} therefore applies, and
$\alpha^1\wedge\alpha^2$ is uniformly distributed on $[0,1]$ under
$\E\langle\cdot\rangle_\lambda$. Integrating in $\lambda$ gives
\eqref{eq:restricted-path-continuity}.
\end{proof}

By Lemma~\ref{lem:restricted-continuity}, the restricted free energies
extend by continuity to $\sfq\in\mathcal Q_1$.

\begin{prop}
\label{prop:remove-constraints}
For every $T,Q<\infty$, there exists
$C=C_{P_1,D,\xi,T,Q}>0$ such that, for all $n\geq1$ and $k\geq2$,
\begin{align}
 \sup_{\substack{0\leq t\leq T,\ \sfq\in\mathcal Q_\infty\\
                  |\sfq|_{L^\infty}\leq Q}}
 |\bar F_{kn}(t,\sfq)-&\bar F_n(t,\sfq)|
 \nonumber \\
 &\quad\leq
 \sup_{\substack{0\leq t\leq T,\ \sfq\in\mathcal Q_\infty\\
                  |\sfq|_{L^\infty}\leq Q}}
 |\bar F_{1,k,n}(t,\sfq;\Sigma^{\mathrm{bal}}_{kn})
  -\bar F_{0,k,n}(t,\sfq;\Sigma^{\mathrm{bal}}_{kn})|
 +C\frac{\log(n+1)}{\sqrt n}.
\label{eq:constraint-reduction}
\end{align}
\end{prop}

This follows by the triangle inequality from the two endpoint
estimates below.

\begin{lem}
\label{lem:full-endpoint}
For every $T,Q<\infty$, there exists
$C=C_{P_1,D,\xi,T,Q}>0$ such that, for all $n\geq1$ and $k\geq2$,
\[
 \sup_{\substack{0\leq t\leq T,\ \sfq\in\mathcal Q_\infty\\
                   |\sfq|_{L^\infty}\leq Q}}
 |\bar F_{1,k,n}(t,\sfq;\Sigma^{\mathrm{bal}}_{kn})
       -\bar F_{kn}(t,\sfq)|
 \leq C\frac{\log(n+1)}{\sqrt n}.
\]
\end{lem}

\begin{lem}
\label{lem:block-endpoint}
For every $T,Q<\infty$, there exists
$C=C_{P_1,D,\xi,T,Q}>0$ such that, for all $n\geq1$ and $k\geq2$,
\[
 \sup_{\substack{0\leq t\leq T,\ \sfq\in\mathcal Q_\infty\\
                   |\sfq|_{L^\infty}\leq Q}}
 |\bar F_{0,k,n}(t,\sfq;\Sigma^{\mathrm{bal}}_{kn})
       -\bar F_n(t,\sfq)|
 \leq C\sqrt{\frac{\log(n+1)}n}.
\]
\end{lem}

\subsection{The Parisi equation}
\label{sec:parisi-equation}

Conditionally on the Hamiltonian, the free energy is the averaged
initial value of a deterministic Parisi equation. Here $s$ is the
cascade parameter, $x$ is the external field, and $t$ enters through
the terminal condition.

\begin{lem}
\label{lem:parisi-representation}
Let $\mu$ be a nonzero finite measure on $\R^{D\times N}$, supported on
configurations with $|\sigma_{\cdot,i}|\leq1$ for every $i$. Let
$\sfq\in C^1([0,1];S^D_+)$ be increasing. The equation
\begin{equation}
\begin{cases}
 \partial_s\Phi+
 \displaystyle\sum_{i=1}^N\dot\sfq(s)\cdot\nabla^2_{x_{\cdot,i}}\Phi
 +s\,\dot\sfq(s)\cdot
       \bigl(\nabla_x\Phi(\nabla_x\Phi)^\intercal\bigr)=0,
 & s<1,\\[1ex]
 \displaystyle
 \Phi(1,x)=\log\int
 e^{x\cdot\sigma-\sfq(1)\cdot\sigma\sigma^\intercal}\dd\mu(\sigma)
\end{cases}
\label{eq:parisi-pde}
\end{equation}
has a unique classical solution with bounded spatial gradient.
It is convex in $x$, its spatial derivatives of positive order
are bounded, and
\[
 |\nabla_{x_{\cdot,i}}\Phi(s,x)|\leq1,\qquad 1\leq i\leq N.
\]
Moreover,
\begin{equation}
 \E\Phi(0,\sqrt{2\sfq(0)}z_0)
 =\E\log\iint
 e^{\sqrt2 W_N^{\sfq}(\alpha)\cdot\sigma
       -\sfq(1)\cdot\sigma\sigma^\intercal}
 \dd\mu(\sigma)\dd\mathfrak R(\alpha),
 \label{eq:parisi-representation}
\end{equation}
where $z_0$ has independent standard Gaussian entries.
\end{lem}

\begin{proof}
Apply \cite[Lemma~4.5 and Proposition~4.7]{ChenMourratSimultaneous}
in dimension $DN$, with covariance path $I_N\otimes\sfq$ and
distribution function $s\mapsto s$. Adding the independent field
$\sqrt{2\sfq(0)}z_0$ gives \eqref{eq:parisi-representation}, as in
\cite[Lemma~5.4]{ChenMourratSimultaneous}. Differentiating the
representation gives the bound on the gradient in each column, and
H\"older's inequality gives convexity. For uniqueness, the difference
of two solutions satisfies a linear backward equation with bounded
drift, zero terminal value, and at most linear growth, and It\^o's
formula with localization shows that it vanishes.
\end{proof}

For the restricted model of size $N$, take
\[
 \dd\mu(\sigma)=\mathbf1_\Sigma(\sigma)
 \exp\left(\sqrt{2t}H_N(\sigma)
       -tN\xi\left(\frac{\sigma\sigma^\intercal}{N}\right)\right)
 \dd P_N(\sigma).
\]
Then
\[
 \bar F_N(t,\sfq;\Sigma)
 =-\frac1N\E_{H_N,z_0}
 \Phi(0,\sqrt{2\sfq(0)}z_0;\Sigma).
\]
For the interpolated model, replace $H_N$ by $H_{\theta,k,n}$, $N$ by
$kn$, and the variance correction accordingly. The expectation over the
Hamiltonian is taken outside the PDE. Restricting the domain decreases
the integral, so restriction increases the free energy. If the terminal
function is a sum of functions of disjoint blocks of columns, then the
sum of the solutions for the blocks solves \eqref{eq:parisi-pde}, since
the nonlinear term is a sum over columns. By uniqueness, it is the
solution.

We next record the response to a terminal perturbation. If the terminal
exponent contains $h\cdot b(\sigma)$, with $b$ a bounded Borel map to a
finite-dimensional Euclidean space, then the mixed derivatives in $x$
and $h$ are bounded. This is
\cite[Proposition~4.7]{ChenMourratSimultaneous}, applied with zero
covariance in the $h$ coordinates.

Let $B$ be a standard $D\times N$ Brownian motion independent of $z_0$,
and let
\begin{equation}
 \dd X_s
 =2s\,\dot\sfq(s)\nabla_x\Phi(s,X_s)\dd s
 +\sqrt{2\dot\sfq(s)}\dd B_s,
 \qquad X_0=\sqrt{2\sfq(0)}z_0.
 \label{eq:parisi-diffusion}
\end{equation}
Since the Hessian of $\Phi$ is bounded, this equation has a unique
strong solution. Let $(\mathcal G_s)_{0\leq s\leq1}$ be the filtration
generated by $z_0$ and $B$. Conditionally on $\mathcal G_1$, sample
$\sigma$ with law
\[
 \exp\bigl(X_1\cdot\tau-\sfq(1)\cdot\tau\tau^\intercal
                  -\Phi(1,X_1)\bigr)\dd\mu(\tau).
\]
We also write $\E$ for the expectation under this construction.
Hessians and covariance matrices act on the vector of stacked columns
in $\R^{DN}$.

\begin{lem}
\label{lem:martingale}
Let $v$ be a bounded Borel function on $\supp\mu$, and let $\Phi^y$ be
the solution with terminal measure $e^{yv}\mu$. Along the unperturbed
diffusion,
\begin{equation}
\begin{aligned}
 \left.\partial_y\Phi^y(s,X_s)\right|_{y=0}
 &=\E[v(\sigma)\,|\,\mathcal G_s],\\
 \left.\partial_y^2\Phi^y(s,X_s)\right|_{y=0}
 &=s\Var(v(\sigma)\,|\,\mathcal G_s)
   +\E\left[\int_s^1\Var(v(\sigma)\,|\,\mathcal G_r)\dd r
                \,|\,\mathcal G_s\right].
\end{aligned}
\label{eq:terminal-response}
\end{equation}
In particular, set $m_s=\E[\sigma\,|\,\mathcal G_s]$,
$V_s=\Cov(\sigma\,|\,\mathcal G_s)$, and $K_s=\nabla_x^2\Phi(s,X_s)$.
Then $m_s=\nabla_x\Phi(s,X_s)$ and
\begin{equation}
 K_s=sV_s+
 \E\left[\int_s^1 V_r\dd r\,|\,\mathcal G_s\right].
 \label{eq:covariance-identity}
\end{equation}
If $\dot\sfq\succeq\varepsilon I_D$ for some $\varepsilon>0$, then
every deterministic orthogonal projection $\Pi$ on $\R^{DN}$ satisfies
\begin{equation}
 \int_0^1\E\tr(\Pi K_s^2)\dd s
 \leq\frac{\E|\Pi\sigma|^2}{2\varepsilon}.
 \label{eq:parisi-covariance-energy}
\end{equation}
\end{lem}

\begin{proof}
Put $\chi_1=\left.\partial_y\Phi^y\right|_{y=0}$ and
$\chi_2=\left.\partial_y^2\Phi^y\right|_{y=0}$, and differentiate the
PDE in $y$, keeping $X$ unperturbed. By It\^o's formula,
\[
 \dd\chi_1(s,X_s)
 =\sum_{i=1}^N\nabla_{x_{\cdot,i}}\chi_1(s,X_s)\cdot
       \sqrt{2\dot\sfq(s)}\dd B_{\cdot,i,s},
\]
while the drift of $\chi_2(s,X_s)$ is
\[
 -2s\sum_{i=1}^N
 |\sqrt{\dot\sfq(s)}\nabla_{x_{\cdot,i}}\chi_1(s,X_s)|^2.
\]
The terminal values are the conditional mean and variance of
$v(\sigma)$ given $\mathcal G_1$. Hence $\chi_1(s,X_s)$ is the
conditional mean of $v(\sigma)$ given $\mathcal G_s$. It follows that
the drift of $\Var(v(\sigma)\,|\,\mathcal G_s)$ is the same expression
without the factor $s$, and therefore that
\[
 \chi_2(s,X_s)-s\Var(v(\sigma)\,|\,\mathcal G_s)
 +\int_0^s\Var(v(\sigma)\,|\,\mathcal G_r)\dd r
\]
is a martingale. Its terminal value gives
\eqref{eq:terminal-response}. The bounds on the derivatives of $\Phi^y$
justify the differentiation in $y$ and the use of It\^o's formula.

For $v(\sigma)=a\cdot\sigma$, translation in $x$ gives
$\Phi^y(s,x)=\Phi(s,x+ya)$. Varying $a$ yields
\eqref{eq:covariance-identity} and
\begin{equation}
 \dd m_s=K_s\bigl(I_N\otimes\sqrt{2\dot\sfq(s)}\bigr)\dd B_s.
 \label{eq:mean-martingale}
\end{equation}
By It\^o's isometry and Jensen's inequality,
\[
 2\varepsilon\int_0^1\E\tr(\Pi K_s^2)\dd s
 \leq\E|\Pi m_1|^2-\E|\Pi m_0|^2
 \leq\E|\Pi\sigma|^2,
\]
as desired.
\end{proof}

The same identities hold when the diffusion starts at a deterministic
$(s,x)$. The second derivative in \eqref{eq:terminal-response} is
nonnegative. Applying it to every linear combination of $\sigma$ and
$b(\sigma)$ shows that $\Phi(s,x;h)$ is jointly convex in $(x,h)$. In
particular, $\Phi$ is jointly convex in the external field and the
Gaussian disorder, which is used in Section~\ref{sec:overlap-proof}; no
convexity of $\xi$ is required. The first identity in
\eqref{eq:terminal-response} also shows that the derivative in a
disorder coordinate is a conditional average of the corresponding
terminal coefficient.

In the next corollary, $\langle\cdot\rangle$ denotes the Gibbs measure
in \eqref{eq:parisi-representation}. Replicas are independent
conditionally on the cascade and the Gaussian field.

\begin{cor}
\label{cor:correlation-identity}
For every Borel function $v$ bounded on $\supp\mu$,
\begin{equation}
 \E\langle v(\sigma^1)v(\sigma^2)\rangle
 =\int_0^1\E\left[\E[v(\sigma)\,|\,\mathcal G_s]^2\right]\dd s.
 \label{eq:scalar-replicas}
\end{equation}
\end{cor}

\begin{proof}
Differentiate \eqref{eq:parisi-representation} for the terminal
measure $e^{yv}\mu$. By \eqref{eq:terminal-response} at $s=0$,
\[
 \E\langle v(\sigma)\rangle=\E v(\sigma),\qquad
 \E\bigl[\langle v(\sigma)^2\rangle-\langle v(\sigma)\rangle^2\bigr]
 =\int_0^1\E\Var(v(\sigma)\,|\,\mathcal G_s)\dd s.
\]
Apply the first identity to $v^2$ and expand the conditional variance
in the second. The terms involving one replica cancel.
\end{proof}

\subsection{Proof of Lemma~\ref{lem:full-endpoint}}
\label{sec:full-endpoint}

Fix $0\leq t\leq T$ and $|\sfq|_{L^\infty}\leq Q$, and write $N=kn$ and
$\supp P_1=\{\tau_1,\ldots,\tau_M\}$. At $\theta=1$, the unrestricted
model is the model of size $N$. Restriction therefore gives
$\bar F_{1,k,n}(t,\sfq;\Sigma^{\mathrm{bal}}_{kn})\geq\bar F_N(t,\sfq)$.
In this proof, $\langle\cdot\rangle$ denotes the unrestricted Gibbs
measure of size $N$.

For each $\sigma$, keep its first block and change the fewest possible
spins in each other block to obtain a configuration $\sigma'$ with
$L_\ell(\sigma')=L_1(\sigma)$ for every $\ell$. The choice of
$\sigma'$ is deterministic and does not depend on the disorder. The
number of changed sites is
\[
 d(\sigma):=\#\{i:\sigma_{\cdot,i}\ne\sigma'_{\cdot,i}\}
 =\frac n2\sum_{\ell=2}^k\sum_{m=1}^M
 |L_\ell(\sigma)(\{\tau_m\})-L_1(\sigma)(\{\tau_m\})|.
\]
The law of $\sigma$ under $\E\langle\cdot\rangle$ is invariant under
permutations of the $N$ sites. Hence, conditionally on the empirical
distribution $L$ of all $N$ spins, it is uniform on the
configurations with empirical distribution $L$. Sampling without
replacement gives
\[
 \Var(L_\ell(\sigma)(\{\tau_m\})\,|\,L)
 \leq\frac{L(\{\tau_m\})}{n},
\]
and every block has conditional mean $L(\{\tau_m\})$. By Jensen's
inequality and $\Var(X-Y)\leq2\Var X+2\Var Y$,
\[
 \E\bigl[|L_\ell(\sigma)(\{\tau_m\})-L_1(\sigma)(\{\tau_m\})|
 \,|\,L\bigr]
 \leq2\sqrt{\frac{L(\{\tau_m\})}{n}}.
\]
Summing over $m$ and $\ell$ and using
$\sum_m\sqrt{L(\{\tau_m\})}\leq\sqrt M$, we obtain
\begin{equation}
 \E\langle d(\sigma)\rangle\leq(k-1)\sqrt{Mn}.
 \label{eq:full-endpoint-changed-sites}
\end{equation}

We now estimate the change in the Gibbs weight. Conditionally on the
cascade, the centered part of $H_N^{t,\sfq}$ has covariance
\[
 \mathcal C(\sigma,\alpha;\tau,\beta)
 =2tN\xi\left(\frac{\sigma\tau^\intercal}{N}\right)
  +2\sfq(\alpha\wedge\beta)\cdot\sigma\tau^\intercal.
\]
Since all columns have norm at most $1$,
\[
 \left|\frac{\sigma\tau^\intercal-\sigma'\tau^\intercal}{N}\right|
 \leq\frac{2d(\sigma)}N,\qquad
 |\mathcal C(\sigma,\alpha;\tau,\beta)
   -\mathcal C(\sigma',\alpha;\tau,\beta)|
 \leq Cd(\sigma),
\]
with $C$ depending only on $D,\xi,T,Q$. Let $(\tau,\beta)$ be an
independent replica. Gaussian integration by parts, including the
variance correction, gives
\[
\begin{aligned}
 &\E\langle H_N^{t,\sfq}(\sigma,\alpha)
             -H_N^{t,\sfq}(\sigma',\alpha)\rangle\\
 &=\frac12\E\bigl\langle
       \mathcal C(\sigma,\alpha;\sigma,\alpha)
       +\mathcal C(\sigma',\alpha;\sigma',\alpha)
       -2\mathcal C(\sigma,\alpha;\sigma',\alpha)\bigr\rangle\\
 &\quad-\E\bigl\langle
       \mathcal C(\sigma,\alpha;\tau,\beta)
       -\mathcal C(\sigma',\alpha;\tau,\beta)\bigr\rangle.
\end{aligned}
\]
The first term on the right is one half of the variance of an
increment. By the covariance bound, both terms are at most
$C\E\langle d(\sigma)\rangle$. Every atom of $P_1$ has positive mass,
so $\log(P_N(\{\sigma\})/P_N(\{\sigma'\}))\leq Cd(\sigma)$, with $C$
now depending also on $P_1$. Thus
\begin{equation}
 \E\left\langle
 H_N^{t,\sfq}(\sigma,\alpha)-H_N^{t,\sfq}(\sigma',\alpha)
 +\log\frac{P_N(\{\sigma\})}{P_N(\{\sigma'\})}
 \right\rangle
 \leq C\E\langle d(\sigma)\rangle.
 \label{eq:full-endpoint-weight-change}
\end{equation}

Fix the disorder and the cascade, and write $Z$ and $Z_{\mathrm{bal}}$
for the unrestricted and balanced partition functions. We use the law
of $(\sigma',\alpha)$ under the Gibbs measure as a trial measure for
$Z_{\mathrm{bal}}$. The Gibbs variational principle gives
\[
 \log\frac Z{Z_{\mathrm{bal}}}
 \leq\left\langle
 H_N^{t,\sfq}(\sigma,\alpha)-H_N^{t,\sfq}(\sigma',\alpha)
 +\log\frac{P_N(\{\sigma\})}{P_N(\{\sigma'\})}
 \right\rangle
 +\mathrm H(\sigma\,|\,\sigma',\alpha),
\]
where $\mathrm H(\cdot\,|\,\cdot)$ is the conditional Shannon entropy
under the Gibbs measure with fixed disorder. Indeed, the marginal law
of $\alpha$ is the same for $(\sigma,\alpha)$ and $(\sigma',\alpha)$,
so the relative entropies with respect to $\mathfrak R$ cancel. Since
$\sigma'$ is a function of $\sigma$, the remaining entropy difference
is $\mathrm H(\sigma\,|\,\alpha)-\mathrm H(\sigma'\,|\,\alpha)
=\mathrm H(\sigma\,|\,\sigma',\alpha)$.

Given $\sigma'$, the configuration $\sigma$ is determined by the set of
changed sites and the original spins at these sites. Put
\[
 \bar p=\frac1N\E\langle d(\sigma)\rangle,\qquad
 h(p)=-p\log p-(1-p)\log(1-p),
\]
with $0\log0=0$. By the chain rule for entropy and the concavity of
$h$,
\[
 \E\mathrm H(\sigma\,|\,\sigma',\alpha)
 \leq\E\sum_{i=1}^N
 h\bigl(\langle\mathbf1_{\{\sigma_{\cdot,i}\ne\sigma'_{\cdot,i}\}}\rangle\bigr)
 +\E\langle d(\sigma)\rangle\log M
 \leq Nh(\bar p)+N\bar p\log M.
\]
Together with \eqref{eq:full-endpoint-weight-change}, this gives
\[
 0\leq
 \bar F_{1,k,n}(t,\sfq;\Sigma^{\mathrm{bal}}_{kn})-\bar F_N(t,\sfq)
 \leq C\bar p+h(\bar p)+\bar p\log M.
\]
By \eqref{eq:full-endpoint-changed-sites}, $\bar p\leq\sqrt{M/n}$. If
$n\geq M$, we use $h(p)\leq p\log(e/p)$ and the fact that
$p\mapsto p\log(e/p)$ is increasing on $(0,1]$ to obtain the bound
$C\log(n+1)/\sqrt n$. If $n<M$, we use $h\leq\log2$ and enlarge $C$.
All constants are independent of $n$ and $k$.

\subsection{Proof of Lemma~\ref{lem:block-endpoint}}
\label{sec:block-endpoint}

Fix $0\leq t\leq T$ and $|\sfq|_{L^\infty}\leq Q$, and suppose first
that $\sfq$ is smooth. At $\theta=0$, the terminal function of the
unrestricted model is a sum of functions of disjoint blocks. By the
additivity of the Parisi equation,
\[
 \bar F_{0,k,n}(t,\sfq)=\bar F_n(t,\sfq),\qquad
 \bar F_{0,k,n}(t,\sfq;\Sigma^{\mathrm{bal}}_{kn}(L))
 =\bar F_n(t,\sfq;\Sigma_n(L)).
\]
Since $\Sigma^{\mathrm{bal}}_{kn}(L)\subseteq\Sigma^{\mathrm{bal}}_{kn}$,
\begin{equation}
 \bar F_n(t,\sfq)
 \leq\bar F_{0,k,n}(t,\sfq;\Sigma^{\mathrm{bal}}_{kn})
 \leq\min_{L\in\mathcal L_n}\bar F_n(t,\sfq;\Sigma_n(L)).
 \label{eq:block-sandwich}
\end{equation}
We shall prove that, for $0<\eta\leq1/2$,
\begin{equation}
 0\leq\min_{L\in\mathcal L_n}\bar F_n(t,\sfq;\Sigma_n(L))
       -\bar F_n(t,\sfq)
 \leq C\eta+\frac{\log|\mathcal L_n|}{n\eta}.
 \label{eq:type-bound}
\end{equation}

Note that $g\mapsto H_n(\sigma;g)$ is linear.
From now on, let $g$ be a standard Gaussian vector.
Let $\Phi$ and $\Phi_L$ be the Parisi solutions for the
full and restricted models. Their terminal values satisfy
$\Phi(1,x;g)=\log\sum_L e^{\Phi_L(1,x;g)}$. For $s>0$, set
\[
 \widehat\Phi(s,x;g)=\frac1s\log\sum_{L\in\mathcal L_n}
 e^{s\Phi_L(s,x;g)},\qquad
 p_L=\frac{e^{s\Phi_L}}{\sum_{L'}e^{s\Phi_{L'}}}.
\]
The factor $s$ in the exponent is chosen so that the quadratic terms in
the gradients cancel. A direct computation gives
\[
 \partial_s\widehat\Phi
 +\sum_{i=1}^n\dot\sfq\cdot\nabla^2_{x_{\cdot,i}}\widehat\Phi
 +s\,\dot\sfq\cdot\bigl(\nabla_x\widehat\Phi(\nabla_x\widehat\Phi)^\intercal\bigr)
 =\frac1{s^2}\sum_L p_L\log p_L\leq0.
\]
The right side is minus the entropy of the weights $(p_L)$, divided
by $s^2$. Since $\widehat\Phi(1)=\Phi(1)$, a comparison argument on
$[\eta,1]$ gives
\[
 \Phi(\eta,x;g)
 \leq\frac1\eta\log\sum_L e^{\eta\Phi_L(\eta,x;g)}.
\]
Indeed, $\widehat\Phi-\Phi$ satisfies a linear backward equation with
bounded drift, zero terminal value, and a nonpositive source term,
which is bounded on $[\eta,1]$. It\^o's formula with localization shows
that $\widehat\Phi-\Phi\geq0$.

To pass from time $\eta$ to time $0$, let $B$ be a standard
$D\times n$ Brownian motion independent of $(g,z_0)$, and set
\[
 Y_s=\sqrt{2\sfq(0)}z_0+
       \int_0^s\sqrt{2\dot\sfq(r)}\dd B_r.
\]
For $\Psi=\Phi$ or $\Psi=\Phi_L$, It\^o's formula along this
driftless process gives
\[
 0\leq\E\Psi(0,Y_0;g)-\E\Psi(\eta,Y_\eta;g)\\
 =\E\int_0^\eta r\sum_{i=1}^n
 |\sqrt{\dot\sfq(r)}\nabla_{x_{\cdot,i}}\Psi(r,Y_r;g)|^2\dd r
 \leq n\eta\tr\sfq(1).
\]
This estimate uses the bound on $\sfq$ but not on its derivative.

The random variable $Y_\eta$ has the law of $\sqrt{2\sfq(\eta)}z$,
where $z$ is standard Gaussian and independent of $g$. The function
$(g,z)\mapsto\Phi_L(\eta,\sqrt{2\sfq(\eta)}z;g)$ is Lipschitz, with
squared Lipschitz constant at most
\[
 2n\Bigl(t\sup_{|a|\leq1}|\xi(a)|+Q\Bigr).
\]
The derivative in $z$ is controlled by the bound on the gradient in
each column. By \eqref{eq:terminal-response}, the derivative in $g$ is
a conditional average of terminal coefficient vectors, each of squared
norm at most $2tn\sup_{|a|\leq1}|\xi(a)|$. Jensen's inequality and
Gaussian concentration give
\[
 \E\left[\frac1\eta\log\sum_L
                 e^{\eta\Phi_L(\eta,Y_\eta;g)}\right]
 \leq\max_L\E\Phi_L(\eta,Y_\eta;g)
       +\frac{\log|\mathcal L_n|}{\eta}+Cn\eta.
\]
Combining these bounds and using $\tr\sfq(1)\leq\sqrt D\,Q$, we obtain
\begin{equation}
 \E\Phi(0,Y_0;g)
 \leq\max_L\E\Phi_L(0,Y_0;g)
       +\frac{\log|\mathcal L_n|}{\eta}+Cn\eta.
 \label{eq:initial-comparison}
\end{equation}
By the Parisi representation, this is \eqref{eq:type-bound}. Since
$|\mathcal L_n|\leq(n+1)^{M}$ with $M=\#\supp P_1$, the choice
$\eta=\tfrac12\sqrt{\log(n+1)/n}$ proves the lemma for smooth
$\sfq$. A bounded path in $\mathcal Q_\infty$ is the limit in $L^1$ of
smooth increasing paths with the same bound $Q$, and
Lemma~\ref{lem:restricted-continuity} extends the result to all such
paths.

\subsection{Reduction to block overlaps}
\label{sec:block-overlaps}

Fix $0\leq t\leq T$, and suppose first that $\sfq$ is smooth,
increasing, and bounded by $Q$. For $0<\varepsilon\leq1$, set
\[
 \sfq^\varepsilon(s)=\sfq(s)+\varepsilon(1+s)I_D.
\]
Then $\dot\sfq^\varepsilon\succeq\varepsilon I_D$ and
\begin{equation}
 \int_0^1\|\sfq^\varepsilon(s)-\sfq(s)\|_{\mathrm{op}}\dd s
 =\frac32\varepsilon.
 \label{eq:q-eps-bound}
\end{equation}
By Lemma~\ref{lem:restricted-continuity},
\begin{equation}
 |\bar F_{1,k,n}(t,\sfq;\Sigma^{\mathrm{bal}}_{kn})
       -\bar F_{0,k,n}(t,\sfq;\Sigma^{\mathrm{bal}}_{kn})|
       \leq3\varepsilon+\int_0^1
 \left|\frac{\dd}{\dd\theta}\bar F_{\theta,k,n}
       (t,\sfq^\varepsilon;\Sigma^{\mathrm{bal}}_{kn})\right|\dd\theta.
\label{eq:two-terms}
\end{equation}

For two replicas $\sigma^1,\sigma^2$, write
\[
 \overlap^{(\ell)}=\frac1n
 \sigma^{1,(\ell)}(\sigma^{2,(\ell)})^\intercal,\qquad
 \bar\overlap=\frac1k\sum_{\ell=1}^k\overlap^{(\ell)}.
\]
For any restriction $\Sigma$ of positive reference measure mass, the variance
correction in \eqref{eq:restricted-hamiltonian} cancels the diagonal
terms in Gaussian integration by parts. Thus, for $0<\theta<1$,
\begin{equation}
 \frac{\dd}{\dd\theta}\bar F_{\theta,k,n}(t,\sfq;\Sigma)
 =t\E\Bigl\langle
 \xi(\bar\overlap)-\frac1k\sum_{\ell=1}^k\xi(\overlap^{(\ell)})
 \Bigr\rangle_{\theta,k,n}^{t,\sfq;\Sigma}.
 \label{eq:interpolation-derivative}
\end{equation}
Here the brackets denote the Gibbs measure with the indicated
restriction. For $\Sigma=\Sigma^{\mathrm{bal}}_{kn}$, we abbreviate them
to $\langle\cdot\rangle_{\theta,k,n}^{t,\sfq}$.

The linear terms in Taylor's formula at $\bar\overlap$ cancel, so
\[
 \Bigl|\xi(\bar\overlap)-\frac1k\sum_{\ell=1}^k\xi(\overlap^{(\ell)})\Bigr|
 \leq\frac1{2k}
 \Bigl(\sup_{|a|\leq1}\|\nabla^2\xi(a)\|_{\mathrm{op}}\Bigr)
 \sum_{\ell=1}^k|\overlap^{(\ell)}-\bar\overlap|^2.
\]
Since
\[
 \sum_{\ell=1}^k|\overlap^{(\ell)}-\bar\overlap|^2
 =\frac1{2k}\sum_{\ell,\ell'=1}^k
 |\overlap^{(\ell)}-\overlap^{(\ell')}|^2
\]
and the balanced model is invariant under permutations of the blocks,
\begin{equation}
 \Bigl|\frac{\dd}{\dd\theta}\bar F_{\theta,k,n}
                (t,\sfq;\Sigma^{\mathrm{bal}}_{kn})\Bigr|
 \leq\frac t4
 \Bigl(\sup_{|a|\leq1}\|\nabla^2\xi(a)\|_{\mathrm{op}}\Bigr)
 \E\langle|\overlap^{(1)}-\overlap^{(2)}|^2\rangle_{\theta,k,n}^{t,\sfq}.
 \label{eq:derivative-bound}
\end{equation}

We express the last expectation using the Parisi diffusion. Realize
$H_{\theta,k,n}(\sigma;g)$ on
$\Sigma^{\mathrm{bal}}_{kn}\cap(\supp P_1)^{kn}$ by the positive
square root of its covariance matrix and a standard Gaussian vector
$g$. Conditionally on $g$, we use \eqref{eq:parisi-pde} with $N=kn$
and terminal value
\begin{equation}
\begin{aligned}
 \Phi(1,x;g)=\log\int_{\Sigma^{\mathrm{bal}}_{kn}}
 \exp\biggl(&\sqrt{2t}\,H_{\theta,k,n}(\sigma;g)+x\cdot\sigma\\
 &-tkn\,\xi\left(\frac{\sigma\sigma^\intercal}{kn}\right)
 -\sfq(1)\cdot\sigma\sigma^\intercal\biggr)\dd P_{kn}(\sigma).
\end{aligned}
\label{eq:balanced-terminal}
\end{equation}
We take $z_0$ and $B$ in \eqref{eq:parisi-diffusion} independent of
$g$, and sample the terminal configuration as in
Section~\ref{sec:parisi-equation}. Let $(\mathcal G_s)$ be the
filtration generated by $g$, $z_0$ and $B$, and denote the joint law
and its expectation by $\mathbb P_\theta$ and $\E_\theta$.

For each $s$, let $\sigma_s^1,\sigma_s^2$ be independent samples from
the conditional law of $\sigma$ given $\mathcal G_s$, and set
$\overlap_s^{(\ell)}=n^{-1}\sigma_s^{1,(\ell)}(\sigma_s^{2,(\ell)})^\intercal$.
Expectations include these conditional samples. We never need a
coupling of the samples at different times.

\begin{lem}
\label{lem:replicas}
We have
\begin{equation}
 \E\langle|\overlap^{(1)}-\overlap^{(2)}|^2\rangle_{\theta,k,n}^{t,\sfq}
 =\int_0^1\E_\theta|\overlap_s^{(1)}-\overlap_s^{(2)}|^2\dd s.
 \label{eq:replica-identity}
\end{equation}
\end{lem}

\begin{proof}
Conditionally on $g$, apply Corollary~\ref{cor:correlation-identity}
to $v(\sigma)=\sigma^{(\ell)}_{\cdot,i}\cdot\sigma^{(\ell')}_{\cdot,j}$
for $1\leq i,j\leq n$ and $\ell,\ell'\in\{1,2\}$. Average over $g$ and
sum with coefficients $n^{-2}(-1)^{\ell+\ell'}$. Expanding the squared
Frobenius norm and using the conditional independence of the replicas
gives the identity.
\end{proof}

The remaining estimate is proved in Section~\ref{sec:overlap-proof}.

\begin{prop}
\label{prop:block-overlap}
Suppose that $P_1$ has finite support in the unit ball of $\R^D$. Fix
$\varepsilon>0$, and suppose that $\sfq\in C^1([0,1];S^D_+)$ satisfies
$\dot\sfq(s)\succeq\varepsilon I_D$ for every $s\in[0,1]$. There exists
$C_D>0$, depending only on $D$, such that for all $n\geq1$, $k\geq2$,
$t\geq0$, and $\theta\in[0,1]$,
\begin{equation}
 \int_0^1\E_\theta|\overlap_s^{(1)}-\overlap_s^{(2)}|^2\dd s
 \leq C_D(n\varepsilon)^{-1/3}.
 \label{eq:block-overlap-bound}
\end{equation}
\end{prop}

We apply \eqref{eq:derivative-bound}, Lemma~\ref{lem:replicas}, and
Proposition~\ref{prop:block-overlap} to $\sfq^\varepsilon$ in
\eqref{eq:two-terms}. This gives
\[
 |\bar F_{1,k,n}(t,\sfq;\Sigma^{\mathrm{bal}}_{kn})
       -\bar F_{0,k,n}(t,\sfq;\Sigma^{\mathrm{bal}}_{kn})|
\leq
3\varepsilon+\frac{C_Dt}{4}
 \Bigl(\sup_{|a|\leq1}\|\nabla^2\xi(a)\|_{\mathrm{op}}\Bigr)
 (n\varepsilon)^{-1/3}.
\]
We take $\varepsilon=n^{-1/4}$. The resulting bound is $Cn^{-1/4}$,
uniformly in $n$, $k$, $t\leq T$, and $\sfq$. Smooth approximation and
Lemma~\ref{lem:restricted-continuity} extend it to all increasing paths
bounded by $Q$. Proposition~\ref{prop:remove-constraints} and
$\log(n+1)/\sqrt n\leq Cn^{-1/4}$ now prove Proposition~\ref{prop:rate}.

\section{Proof of Proposition~\ref{prop:block-overlap}}
\label{sec:overlap-proof}

Fix the parameters of Proposition~\ref{prop:block-overlap}. Set
$m_s=\E_\theta[\sigma\,|\,\mathcal G_s]$,
$V_s=\Cov_\theta(\sigma\,|\,\mathcal G_s)$ and
$K_s=\nabla_x^2\Phi(s,X_s;g)$, and let $\Pi$ be the orthogonal
projection onto the coordinates of the first two blocks. The Hessian
$K_s$ acts on all $kn$ sites. Lemma~\ref{lem:martingale}, applied
conditionally on $g$, gives
\begin{equation}
 sV_s\preceq K_s,
 \qquad
 \int_0^1\E_\theta\tr(\Pi K_s^2)\dd s
 \leq\frac n\varepsilon,
 \label{eq:covariance-energy}
\end{equation}
where we used $|\Pi\sigma|^2\leq2n$.

Set $\Delta_s=\overlap_s^{(1)}-\overlap_s^{(2)}$. Then, the law of total variance yields 
the decomposition
\begin{equation}
 \int_0^1\E_\theta|\Delta_s|^2\dd s
 =\int_0^1\E_\theta|\E_\theta[\Delta_s\,|\,\mathcal G_s]|^2\dd s\\
 +\int_0^1\E_\theta
 |\Delta_s-\E_\theta[\Delta_s\,|\,\mathcal G_s]|^2\dd s.
\label{eq:overlap-decomposition}
\end{equation}
We estimate the two terms separately.

\begin{lem}
\label{lem:mean-part}
There exists $C_D>0$, depending only on $D$, such that for every
$0<\delta<1/4$,
\[
 \int_0^1\E_\theta|\E_\theta[\Delta_s\,|\,\mathcal G_s]|^2\dd s
 \leq C_D\left(\delta+\frac1{n\varepsilon\delta^2}\right).
\]
\end{lem}

\begin{lem}
\label{lem:fluctuation-part}
For every $0<\delta<1/4$,
\[
 \int_\delta^1\E_\theta
 |\Delta_s-\E_\theta[\Delta_s\,|\,\mathcal G_s]|^2\dd s
 \leq\frac{4D}{\sqrt{n\varepsilon\delta}}.
\]
\end{lem}

On $[0,\delta]$, the second term in \eqref{eq:overlap-decomposition} is
at most $4\delta$, since $|\Delta_s|\leq2$. The two lemmas
and
\[
 (n\varepsilon\delta)^{-1/2}
 \leq\tfrac12\left(\delta+(n\varepsilon\delta^2)^{-1}\right)
\]
therefore give
\[
 \int_0^1\E_\theta|\Delta_s|^2\dd s
 \leq C_D\left(\delta+\frac1{n\varepsilon\delta^2}\right).
\]
We take $\delta=(n\varepsilon)^{-1/3}$ if this is less than $1/4$;
otherwise we use $|\Delta_s|\leq2$. This proves the
proposition. We prove Lemma~\ref{lem:fluctuation-part} first.

\subsection{Proof of Lemma~\ref{lem:fluctuation-part}}
\label{sec:proof-fluctuation}

Fix $s>0$. Conditionally on $\mathcal G_s$, the replicas
$\sigma_s^1,\sigma_s^2$ are independent. Fix an entry
$\Delta_{s,ij}$ of $\Delta_s$. By conditional independence, the
law of total variance and Jensen's inequality give
\[
\begin{aligned}
 \Var_\theta(\Delta_{s,ij} \,|\,\mathcal G_s)
 &=\E_\theta\bigl[
       \Var_\theta(\Delta_{s,ij} \,|\,\mathcal G_s,\sigma_s^1)
       \,|\,\mathcal G_s\bigr]
   +\Var_\theta\bigl(
       \E_\theta[\Delta_{s,ij} \,|\,\mathcal G_s,\sigma_s^1]
       \,|\,\mathcal G_s\bigr)\\
 &\leq\E_\theta\bigl[
       \Var_\theta(\Delta_{s,ij} \,|\,\mathcal G_s,\sigma_s^1)
       +\Var_\theta(\Delta_{s,ij} \,|\,\mathcal G_s,\sigma_s^2)
       \,|\,\mathcal G_s\bigr].
\end{aligned}
\]

Fix $\sigma_s^1=\tau$. Then $\Delta_{s,ij}$ is a linear function of $\sigma_s^2$:
the coefficient of the $j$-th coordinate of site $p$ is $\tau_{i,p}/n$
if $p$ is in the first block, $-\tau_{i,p}/n$ if $p$ is in the second
block, and $0$ otherwise. For a coefficient vector $v$ supported on the
first two blocks, \eqref{eq:covariance-energy} gives
\[
 \Var_\theta(v\cdot\sigma\,|\,\mathcal G_s)
 \leq\frac1s v^\intercal K_sv
 \leq\frac{|v|^2}{s}\sqrt{\tr(\Pi K_s^2)}.
\]
The squared norms of the coefficient vectors, summed over $i,j$, equal
$Dn^{-2}|\Pi\tau|^2\leq2D/n$. The same bound holds with the roles of
the replicas exchanged. Hence,
\[
 \E_\theta\bigl[
 |\Delta_s-\E_\theta[\Delta_s\,|\,\mathcal G_s]|^2\,|\,\mathcal G_s\bigr]
 \leq\frac{4D}{ns}\sqrt{\tr(\Pi K_s^2)}.
\]
Taking expectations and applying the Cauchy--Schwarz inequality yields
\[
 \int_\delta^1\E_\theta
 |\Delta_s-\E_\theta[\Delta_s\,|\,\mathcal G_s]|^2\dd s
 \leq\frac{4D}{n}
 \left(\int_\delta^1\E_\theta\tr(\Pi K_s^2)\dd s\right)^{1/2}
 \left(\int_\delta^1s^{-2}\dd s\right)^{1/2}
 \leq\frac{4D}{\sqrt{n\varepsilon\delta}},
\]
as required.

\subsection{Proof of Lemma~\ref{lem:mean-part}}
\label{sec:proof-mean}

Fix $0<\delta<1/4$. By conditional independence,
\[
 \E_\theta[\Delta_s\,|\,\mathcal G_s]=\frac1n\bigl(m_s^{(1)}(m_s^{(1)})^\intercal
                 -m_s^{(2)}(m_s^{(2)})^\intercal\bigr).
\]
In particular $\E_\theta[\Delta_s\,|\,\mathcal G_s]$ is symmetric, and it suffices to bound
$\int_0^1\E_\theta|\E_\theta[a\cdot\Delta_s\,|\,\mathcal G_s]|^2\dd s$ uniformly over $a\in S^D$
with $|a|=1$.

Let $T_a$ act as $a$ on each site of the first block, as $-a$ on each
site of the second block, and as zero elsewhere. Then $T_a^2\preceq\Pi$
and, by balance, $\sigma^\intercal T_a\sigma=0$. Set
\[
 U(s)=\frac1n\left(\sigma^\intercal T_a m_s
                         -\frac12m_s^\intercal T_a m_s\right)
     =-\frac1{2n}(\sigma-m_s)^\intercal T_a(\sigma-m_s),
\]
and, for $2\delta\leq s\leq1$,
$\overline U(s)=\delta^{-1}\int_{s-\delta}^sU(r)\dd r$. We shall prove
\begin{equation}
 \E_\theta\overline U(s)^2
 \leq\frac1{n^2\delta^2}\int_{s-\delta}^s
 \frac{\E_\theta\tr(\Pi K_r^2)}r\dd r.
 \label{eq:main-variance}
\end{equation}

We first show that $\overline U(s)$ is centered given $\sigma$. Fix
$\tau\in\Sigma^{\mathrm{bal}}_{kn}\cap(\supp P_1)^{kn}$. Exchanging
matching spins in its first two blocks gives a permutation $\pi$ of the
sites with $\pi\tau=\tau$ and $\pi^\intercal T_a\pi=-T_a$. The
covariance of $H_{\theta,k,n}$ and its positive square root commute
with the permutation of configurations induced by $\pi$. Hence
permuting $g$, $z_0$ and $B$ according to $\pi$ preserves the
conditional law given $\sigma=\tau$, and, by uniqueness of the Parisi
solution and of the diffusion, it maps $m_r$ to $\pi m_r$. This
changes the sign of $U(r)$. Consequently,
\[
 \E_\theta[U(r)\,|\,\sigma]=0,
 \qquad \E_\theta[\overline U(s)\,|\,\sigma]=0.
\]

We next write the conditional law given $\sigma=\tau$ as a
log-concave perturbation of a Gaussian law. Let $\Gamma$ be the
product law of $g$, $z_0$ and a standard Brownian motion $B$, and on
this space set
\[
 X_r=\sqrt{2\sfq(0)}z_0+
       \int_0^r\sqrt{2\dot\sfq(u)}\dd B_u.
\]
By It\^o's formula and the Parisi equation,
\[
 \exp\left(\Phi(1,X_1;g)-\int_0^1\Phi(r,X_r;g)\dd r\right)
\]
is the stochastic exponential of the martingale with integrand
$r\sqrt{2\dot\sfq(r)}\nabla_x\Phi(r,X_r;g)$. Since the gradient is
bounded, Novikov's condition (see, Proposition~3.5.12 in \cite{KaratzasShreve1991})
holds conditionally on $(g,z_0)$. By
Girsanov's theorem, under the tilted law the process $X$ solves
\eqref{eq:parisi-diffusion}, and the law of $(g,z_0)$ is unchanged.
Multiplying by the terminal Gibbs density cancels $\Phi(1,X_1;g)$ and
yields
\begin{equation}
 \frac{\dd\mathbb P_\theta(\,\cdot\,|\,\sigma=\tau)}{\dd\Gamma}
 \propto\exp\left(\sqrt{2t}\,H_{\theta,k,n}(\tau;g)+X_1\cdot\tau
                 -\int_0^1\Phi(r,X_r;g)\dd r\right).
 \label{eq:conditional-density}
\end{equation}
Terms depending only on $\tau$ are absorbed in the normalizing
constant. Note that $g$ is still random under this law.

Expand $B$ in a sequence of independent standard Gaussian coordinates,
and let $\mathcal F_m$ be the $\sigma$-field generated by all these
coordinates except the first $m$. Conditionally on $\mathcal F_m$ and
$\sigma=\tau$, let $\zeta$ be the vector formed by $g$, $z_0$ and the
first $m$ coordinates of $B$. Since $X_r$ is affine in $\zeta$,
\eqref{eq:conditional-density} shows that the conditional law of
$\zeta$ has a density proportional to $e^{-\mathcal V(\zeta)}$ with
\[
 \nabla^2\mathcal V
 =I+\int_0^1\nabla_\zeta^2[\Phi(r,X_r;g)]\dd r.
\]
By the joint convexity of $\Phi$ in the field and the disorder, each
integrand is positive semidefinite, so $\nabla^2\mathcal V\succeq I$.

With $\tau$ fixed, the gradient of $U(r)$ with respect to $m_r$ is
$n^{-1}T_a(\tau-m_r)$, and $m_r=\nabla_x\Phi(r,X_r;g)$. Let $w$ be a
direction in the $\zeta$ coordinates. The Cauchy--Schwarz inequality
for the joint Hessian of $\Phi$ in $(x,g)$, applied to the direction
induced by $w$ and to the field direction $T_a(\tau-m_r)$, followed by
the Cauchy--Schwarz inequality in $r$, gives
\[
 |w\cdot\nabla_\zeta\overline U(s)|^2
 \leq\frac{w^\intercal\nabla^2\mathcal V\,w}{n^2\delta^2}
 \int_{s-\delta}^s
 (\tau-m_r)^\intercal T_aK_rT_a(\tau-m_r)\dd r.
\]
Taking $w=(\nabla^2\mathcal V)^{-1}\nabla_\zeta\overline U(s)$ and
applying the Brascamp--Lieb inequality
(see, Theorem~4.1 in \cite{BrascampLieb1976}), we find that the conditional
variance of $\overline U(s)$ given $\mathcal F_m$ and $\sigma=\tau$ is
at most the conditional expectation of the integral on the right,
divided by $n^2\delta^2$. Since the derivatives of $\Phi$ in $(x,g)$
are bounded, $\overline U(s)$ is a smooth function of $\zeta$ and these
computations are justified.

We now remove the conditioning on $\mathcal F_m$. The density in
\eqref{eq:conditional-density} is positive, so the conditional law
given $\sigma=\tau$ is equivalent to $\Gamma$. The $\sigma$-fields
$\mathcal F_m$ decrease to the tail $\sigma$-field of the coordinates
of $B$, which is trivial under $\Gamma$ by Kolmogorov's zero-one law,
and hence also under the conditional law. By the backward martingale
convergence theorem (see, Theorem~4.7.3 in \cite{DurrettPTE}), applied to the
bounded variable $\overline U(s)$, the conditional expectation
$\E_\theta[\overline U(s)\,|\,\mathcal F_m,\sigma=\tau]$ converges to
$\E_\theta[\overline U(s)\,|\,\sigma=\tau]$ almost surely and in $L^2$.
Since $\Var=\E\Var(\cdot\,|\,\mathcal F_m)+\Var\E[\cdot\,|\,\mathcal F_m]$,
the variance bound holds conditionally on $\sigma=\tau$ alone. Since
$\E_\theta[\overline U(s)\,|\,\sigma]=0$, averaging over $\tau$ gives
\[
 \E_\theta\overline U(s)^2
 \leq\frac1{n^2\delta^2}\int_{s-\delta}^s
 \E_\theta\left[
 (\sigma-m_r)^\intercal T_aK_rT_a(\sigma-m_r)
 \right]\dd r.
\]
Conditioning the integrand on $\mathcal G_r$, we obtain
\[
\begin{aligned}
 \E_\theta\left[
 (\sigma-m_r)^\intercal T_aK_rT_a(\sigma-m_r)
 \,|\,\mathcal G_r\right]
 &=\tr\left(K_rT_aV_rT_a\right)\\
 &\leq\frac1r\tr(K_rT_aK_rT_a)
 \leq\frac1r\tr(\Pi K_r^2).
\end{aligned}
\]
The first inequality follows from \eqref{eq:covariance-energy}, and the
second from the Cauchy--Schwarz inequality for the Frobenius inner
product and $T_a^2\preceq\Pi$. This proves \eqref{eq:main-variance}.
Integrating it in $s$ gives
\begin{equation}
 \int_{2\delta}^1\E_\theta\overline U(s)^2\dd s
 \leq\frac1{n^2\delta}\int_\delta^1
 \frac{\E_\theta\tr(\Pi K_r^2)}r\dd r
 \leq\frac1{n\varepsilon\delta^2}.
 \label{eq:averaged-bound}
\end{equation}

It remains to estimate $\E_\theta[\Delta_s\,|\,\mathcal G_s]$. For $r\leq s$, the variable $m_r$ is
$\mathcal G_s$-measurable and $\E_\theta[\sigma\,|\,\mathcal G_s]=m_s$,
so
\begin{equation}
 2\E_\theta[U(r)\,|\,\mathcal G_s]
 =\E_\theta[a\cdot\Delta_s\,|\,\mathcal G_s]-\frac1n(m_s-m_r)^\intercal T_a(m_s-m_r).
 \label{eq:observable-mean}
\end{equation}
We average over $r\in[s-\delta,s]$ and multiply by $\E_\theta[a\cdot\Delta_s\,|\,\mathcal G_s]$.
Young's inequality, $|\E_\theta[a\cdot\Delta_s\,|\,\mathcal G_s]|\leq2$ and
$|v^\intercal T_av|\leq|\Pi v|^2$ give
\[
 \E_\theta|\E_\theta[a\cdot\Delta_s\,|\,\mathcal G_s]|^2
 \leq4\E_\theta\overline U(s)^2
 +\frac4{n\delta}\int_{s-\delta}^s
 \E_\theta|\Pi(m_s-m_r)|^2\dd r.
\]
Under $\mathbb P_\theta$, the process $m$ is a martingale, so
$s\mapsto\E_\theta|\Pi m_s|^2$ is increasing and bounded by $2n$.
Hence
\[
\begin{gathered}
 \E_\theta|\Pi(m_s-m_r)|^2
 =\E_\theta|\Pi m_s|^2-\E_\theta|\Pi m_r|^2,
 \qquad r\leq s,\\
 \int_{2\delta}^1\int_{s-\delta}^s
 \E_\theta|\Pi(m_s-m_r)|^2\dd r\dd s
 \leq n\delta^2.
\end{gathered}
\]
Together with \eqref{eq:averaged-bound} and the bound
$|\E_\theta[a\cdot\Delta_s\,|\,\mathcal G_s]|\leq2$ on $[0,2\delta]$, this yields
\[
 \int_0^1\E_\theta|\E_\theta[a\cdot\Delta_s\,|\,\mathcal G_s]|^2\dd s
 \leq12\delta+\frac4{n\varepsilon\delta^2}.
\]
Summing over an orthonormal basis of $S^D$ proves the lemma.

\printbibliography

\end{document}